\documentclass[11pt,letter]{article}

\usepackage[utf8]{inputenc}
\usepackage{amsmath}
\usepackage{amssymb}
\usepackage{amsfonts}
\usepackage{amscd}
\usepackage{latexsym}
\usepackage{wasysym}
\usepackage{graphicx}
\usepackage{xcolor}
\usepackage[all,cmtip]{xy}
\usepackage{bbold}

\usepackage{amssymb}
\usepackage{amsthm}
\usepackage{enumitem}
\usepackage{mathtools,bbm}
\usepackage{mathtools,bbm}
\usepackage{xpatch}
\usepackage{version,xspace}

\usepackage[margin=1in]{geometry}
\usepackage{hyperref}%
\hypersetup{colorlinks=true, linkcolor=blue, breaklinks=true, urlcolor=blue}
\usepackage{url,doi}

\usepackage[utf8]{inputenc}
\usepackage{amsfonts}
\usepackage{amsthm}
\usepackage{enumitem}
\usepackage{mathtools,bbm}
\usepackage{mathtools,bbm}
\usepackage{xpatch}

\usepackage{lipsum}
\usepackage{mathtools}
\usepackage{cuted}
\usepackage{mathtools,bbm}

\usepackage[caption=false,font=footnotesize]{subfig}
\usepackage[font=small]{caption}
\graphicspath{{./images/}{./fig}}

\usepackage{enumitem}
\setlist{nosep}

\newcommand{\R}{\mathbb{R}}

\newcommand{\until}[1]{\{1,\dots, #1\}}

\newtheorem{theorem}{Theorem}[section]

\newtheorem{lemma}{Lemma}[section]
\newtheorem{proposition}[theorem]{Proposition}
\newtheorem{remark}[theorem]{Remark} 

\newtheorem{conjectures}[theorem]{Conjectures}

{
      \theoremstyle{plain}

}
\newtheorem{definition}{Definition}[section]

\DeclareSymbolFont{bbold}{U}{bbold}{m}{n}
\DeclareSymbolFontAlphabet{\mathbbold}{bbold}
\newcommand{\vect}[1]{\mathbbold{#1}}

\newcommand{\diag}{\operatorname{diag}}

\newcommand{\setdef}[2]{\{#1 \; | \; #2\}}

\usepackage{algpseudocode,algorithm,algorithmicx}

\newcommand{\subscr}[2]{#1_{\textup{#2}}}

\newcommand{\map}[3]{#1: #2 \rightarrow  #3}

\renewcommand{\natural}{{\mathbb{N}}}

\newcommand{\real}{\ensuremath{\mathbb{R}}}

\newcommand{\realnonnegative}{\ensuremath{\mathbb{R}}_{\ge 0}}

\newcommand\norm[1]{\left\lVert#1\right\rVert}
\renewcommand{\norm}[1]{\|#1\|}

\newcommand\oprocendsymbol{\hbox{$\square$}}
\newcommand\oprocend{\relax\ifmmode\else\unskip\hfill\fi\oprocendsymbol}

\DeclareSymbolFont{bbold}{U}{bbold}{m}{n}
\DeclareSymbolFontAlphabet{\mathbbold}{bbold}

\newcounter{saveenum}

\newcommand{\ybar}{\eta}
\renewcommand{\ybar}{\mathbf{1}_B}
\newcommand{\ybard}{\mathbf{1}_{b^*}}

\usepackage[prependcaption,colorinlistoftodos]{todonotes}

\title{Multi-group SIS Epidemics with Simplicial and Higher-Order
  Interactions\footnote{This work was supported in part by the Defense Threat
    Reduction Agency under Contract No.~HDTRA1-19-1-0017.}}

\author{Pedro Cisneros-Velarde,
  Francesco Bullo\footnote{Pedro Cisneros-Velarde (pacisne@gmail.com) and Francesco Bullo
  (bullo@ucsb.edu) are with the Center for Control, Dynamical Systems and
  Computation, University of California, Santa Barbara.}}

\begin{document}
\maketitle

\begin{abstract}
%
  This paper analyzes a Susceptible-Infected-Susceptible (SIS) model of
  epidemic propagation over hypergraphs and, motivated by an important
    special case, we refer to the model as to the \emph{simplicial SIS
      model}.      
 Classically, the multi-group SIS model has assumed pairwise interactions
 of contagion across groups and thus has been vastly studied in the
 literature.  It is only recently that a renewed special attention has been
 drawn to the study of contagion dynamics over higher-order interactions
 and over more general graph structures, like simplexes.
 Previous work on mean-field approximation scalar models of the simplicial SIS
 model has indicated that a new dynamical behavior domain, compared to the
 classical SIS model, appears due to the newly introduced higher order
 interaction terms: both a disease-free equilibrium and an endemic 
 equilibrium co-exist and are both locally asymptotically stable.
 This paper formally establishes that bistability (as a new epidemiological
 behavior) also appears in the multi-group simplicial SIS model. We give sufficient
 conditions over the model's parameters for the appearance of this and the
 other behavioral domains present in the classical multi-group SIS
 model. We additionally provide an algorithm to compute the value of the
 endemic equilibrium and report numerical analysis of the transition from
 the disease-free domain to the bistable domain.
\end{abstract}


\section{Introduction}

The study and modeling of the spread of infectious diseases in contact
networks has a long history of development and is of major relevance
today. A first class of models are called \emph{scalar models}, where a
single population is studied. The epidemiological evolution in this single
population is represented by the dynamics of one or more scalar values
that represent a specific proportion of the population (e.g., a scalar
value can represent the proportion of currently infected people).
We refer to the
work~\cite{HWH:00} for a survey on these type of models. The basic
assumption on these models is that the whole population is homogeneous,
i.e., every individual in the population has the same probability of
interaction.  However, in view of this shortcoming, \emph{network} or
\emph{multi-group models} were introduced, in which several homogeneous
populations, also called groups, interact with each other 
according to an underlying 
contact network. Thus, these models 
can capture different kinds of heterogeneity, e.g., age structures, spatial
diversity and social behavior. 
The epidemics is then modulated by the different 
model parameters (e.g., the recovery rate from a disease) that each
population may have, and the connectivity of the underlying network and the
strength of its connections. 
Thus, 
the propagation of the
epidemic 
is now a 
network process.

Multi-group epidemic models have a longstanding history that can be traced
back to the seminal works~\cite{HWH:78,AL-JAY:76}. A recent interpretation
as an approximation of Markov-chain models is given
by~\cite{FDS-CS-PvM:13}.  Degree-based versions of the model have been
analyzed through statistical mechanics in the physics
community~\cite{RPS-AV:01,AdO:08}.  Stability analyses by the controls
community include~\cite{AF-AI-GS-JJT:07,AK-TB-BG:16}.  Much recent work by
the control community has focused on (i) control of epidemic dynamics in
multi-group models, e.g.~\cite{NJW-CN-GJP:20,CN-VMP-GJP:16}, (ii)
extensions of epidemics on time-varying graphs across populations,
e.g.~\cite{MO-VM-NM:20,PEP-CLB-AN:17}, (iii) extensions to
multi-competitive viruses on multi-group models,
e.g.~\cite{PEP-JL-CLB-AN-TB:17}, and (iv) game-theoretical analysis on
multi-group models, e.g.~\cite{ARH-SS:19,KP-CE-JSW-JSS:17}.  Finally, we
mention the recent surveys~\cite{WM-SM-SZ-FB:16f,CN-VMP-GJP:16}.

In this work, we focus on the \emph{Susceptible-Infected-Susceptible} (SIS)
model for the propagation of infectious diseases in the context of
  social contagion. SIS models are applicable to diseases that have
the possibility of a repeated reinfection, i.e., those in which a person
does not develop permanent immunity after recovery~\cite{MM:15}. Some
examples of these diseases are ghonorrea, chlamydya, the common cold,
etc. In the scalar SIS model, the population can be divided in two
fractions: those who are infected and those who are susceptible to become
infected~\cite{HWH:00}.
In the multi-group SIS model, each node of the graph can be interpreted as
either (i) an individual and its associated scalar variable as the
infection probability, or (ii) as a homogeneous group of individuals and the
associated scalar variable is the fraction of infected individuals. The type of
interaction among the individuals or populations 
defines the social contagion mechanism.

In SIS models, it is important to investigate conditions under which the
system converges or not to a \emph{disease-free} equilibrium, i.e., a state
in which all populations become healthy/uninfected (or equivalently, the
probability of any person of being infected becomes zero) or to an
\emph{endemic} equilibrium, i.e., a state in which all populations maintain
a (nonzero) fraction of its members always infected (or equivalently, the
probability of any person of being infected remains nonzero).

\paragraph*{Nonlinear incidence and simplicial contagion models}

The vast majority of the literature on
multi-group SIS models (and other epidemic models in general) considers
only that the interaction between populations (or individuals) is pairwise,
i.e., the social contagion occurs only through the edges that connect
them. Equivalently, in the context of scalar models, this prevalent
assumption is understood as the \emph{incidence rate}, i.e., the rate of
new infections, being bilinear in the proportions of infected and
susceptible people (because the rate is simply the product of both
  proportions).  The idea of considering 
  nonlinear incidence
rates in epidemic scalar models 
can be traced back to the late eighties~\cite{WML-HWH-SAL:87}.
%

From a network-science viewpoint, the recent work by Iacopini et
al.~\cite{II-GP-AB-VL:19} elaborates on the idea of nonlinear incidence
models and considers higher-orders of interaction in the social contagion
of a disease. Since its publication, the work~\cite{II-GP-AB-VL:19} has
received considerable interest and much attention is now focused on
higher-order interactions and simplicial models. We now elaborate on these
ideas. Consider three populations or individuals $i,j,k$. If the pairwise
interactions $\{i,j\}$ or $\{i,k\}$ occur, then there is a certain
susceptibility of $i$ to be infected. However, if the whole group
$\{i,j,k\}$ interact together, then the likelihood of infection for $i$ may
increase since now the simultaneous interaction effect by $j$ and $k$ are
aggregated to the single pairwise interactions we previously
described. 
We can consider $\{i,j,k\}$ as a hyperedge. 
An important class of hypergraphs is a \emph{simplicial complex}, which is a 
hypergraph that contains all nonempty subsets of hyperedges
as hyperedges.
%
In a simplicial complex, a
  hyperedge with $d$ vertices forms a $(d-1)$-simplex, and the simplicial
  complex is said to be of dimension $d-1$ if $d$ is the largest number of
  vertices in any of its simplices (i.e, in its largest
  simplex). 
As an example, if $\{i,j,k\}$ is a $2$-simplex, then $\{i,j\}$, $\{j,k\}$, $\{i,k\}$, $\{i\}$, $\{j\}$, $\{k\}$ are simplices. Thus, a simplex $\{i,j,k\}$ can be understood as a set of nodes that form a triad. Note that if $\{i,j\}$, $\{j,k\}$ and $\{i,k\}$ belong to a simplicial complex, then $\{i,j,k\}$ is not necessarily a simplex. We refer to~\cite{AH:02} for a general and
    extensive treatment of simplicial complexes. 
Starting from these
ideas, the work~\cite{II-GP-AB-VL:19} proposes a new SIS model that
considers the evolution of the epidemic with an underlying simplicial
complex 
of dimension $2$, as opposed to
the classical SIS model that has up to 1-simplices.
However,~\cite{II-GP-AB-VL:19} performs the analysis of a mean-field
approximation 
which becomes a scalar model. A
different derivation of the SIS model over simplicial complexes was
recently introduced in~\cite{JTM-SG-AA:20} from a Markov-chain and mean-field
approximation perspective up to $2$-simplices. Also recently, Jhun et
al.~\cite{BJ-MJ-BK:19} consider the multi-group SIS model and restrict
their analysis to a mean-field approximation of the model for a special
class of simplicial complexes, namely, an infinite hypergraph composed
  of hyperedges of the same size corresponding to simplicial complexes of
  the same dimension.

As discussed by~\cite{II-GP-AB-VL:19}, the adoption of simplicial
interactions in modeling contagion bears some similarities with the
modeling ideas behind linear threshold models by Granovetter~\cite{MG:78}
in sociology, where individuals adopt innovations only when a certain fraction of
their contacts have earlier adopted that innovation. Moreover, simplicial
and higher-order graphical models may be more accurate than simpler
pairwise contagion models to describe transmission events during large
gatherings or other social aggregation
phenomena~\cite{KFK-LS-DCS-MM:13,GFdA-GP-YM:20}. Overall, the study of simplicial and
higher-order interactions is well motivated by the observation that these
structures are ubiquitous and play an important role in real-world social
networks~\cite{PB-ACH-MJ:04,HH-JT-LL-JL-XU:15,ARB-RA-MTS-AJ-JK:18,MTS-ARB-PH-GL-AJ:20}. We refer to the excellent recent survey~\cite{FBa-GC-II-VL-ML-AP-JGY-GP:20} for an overview of the emerging field of networks with higher-order interactions.

\paragraph*{Problem statement}

We now state what is, to the best of our knowledge, an outstanding open
problem. Namely, no work in the current literature establishes a formal
  analysis of the dynamical behavior of a general multi-group SIS model
with higher-order interaction terms over general classes of
  (hyper)graphs. An example of such model could be an SIS model with
  interactions described by a finite simplicial complex. Our paper
responds to this need.  The analysis of such a model may help better
understand the effect of higher-order interaction terms on the dynamics of
social contagion in societies with large gatherings or other social
aggregation phenomena.

\paragraph*{Contributions}

As main contribution of this paper, we consider the simplicial SIS model
and analyze its dynamical behavior. In particular, we identify conditions
on the parameters of the model that allow us to conclude the existence and
asymptotic behavior of a disease-free and/or endemic equilibrium. 
We prove that the model, according to different regimes in its
parameter space, can have 
%
%
its dynamic behavior classified in three epidemic domains:
(i) \emph{disease-free domain}, (ii) \emph{bistable domain}, and  (iii)
\emph{endemic domain:} (see Definition~\ref{def:epidemic_domains}).
%
%
%
%
While the
conditions
given in our main theorem (Theorem~\ref{th:simplicial_main}) do not exhaust all possible
values of the system parameters, we include numerical results that
illustrate the tightness of our derived conditions. Despite this gap, our
sufficient conditions rigorously establish the crucial qualitative behavior
of transition between the disease-free domain and the bistable domain. To
the best of our knowledge, this transition was formally proved only for the
scalar version of the simplicial SIS model 
in~\cite{II-GP-AB-VL:19}.

As second contribution, we propose an iterative algorithm which 
computes an endemic state through monotone convergence 
when the system is in either the bistable or the endemic domain.  
We remark that obtaining
a closed form expression for an endemic equilibrium appears to be
intractable and, indeed, for the classical multi-group SIS model the
best-known result is a monotonic 
convergent 
iteration, 
see~\cite[Theorem~4.3]{WM-SM-SZ-FB:16f}.

As third contribution, we present a general multi-group SIS model with
higher-order interactions, generalizing the 
two dimensional 
simplicial SIS model. Analyzing
this generalized model, we prove that the existence of the bistable domain
is a general phenomenon resulting from higher-order interactions.  While
the treatment becomes more cumbersome, we show that our analysis techniques
are still applicable.

As minor contributions, we provide numerical examples that illustrate the
behavioral domains of the simplicial SIS model and present two interesting
conjectures about the features of the epidemic diagram. Moreover, we
present a self-contained formal review of 
known results for the
scalar version of the simplicial SIS model to facilitate its comparison with 
the multi-group models.

We conclude by mentioning that, to prove our results, we 
use 
the
theory of Metzler matrices and positive systems, fixed-point analysis of
continuous mappings, and exponential convergence with matrix measures and
Lyapunov theory. 
We review a little known result for exponential
convergence combining the theory of matrix measures 
with the theory of solution estimates (Coppel's inequality) for
systems with continuously differentiable vector fields. We remark that
previous works that analyze the classical multi-group SIS model have used specialized cases of
this 
result, e.g., see~\cite[Theorem~2.7]{AF-AI-GS-JJT:07}.

\begin{figure*}[ht]
  \centering
  \includegraphics[width=.99\linewidth]{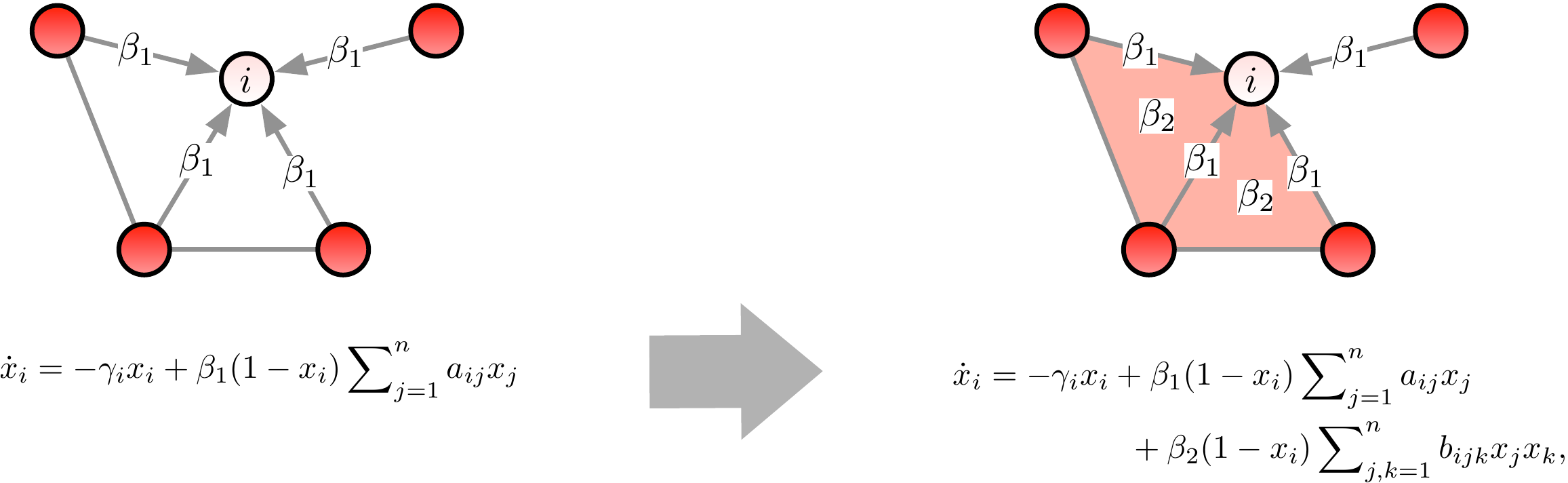}
  \caption{Example: 
    From pairwise to simplicial interactions in the multi-group SIS
    epidemic model. The
    left figure corresponds to the classical version
    and the right one to the simplicial SIS
    model, where hyperedges of three elements provide the higher-order interaction.}\label{fig:comparison-models}
\end{figure*}

\paragraph*{Paper organization} 
Section~\ref{sec:prelim} is the preliminaries and notation. 
Section~\ref{sec:exp-mmeasures} presents a convergence result using matrix measures.
Section~\ref{sec:simplicial} introduces the simplicial SIS model 
and  
Section~\ref{sec:analysis} its 
dynamical analysis. Section~\ref{sec:numeric} presents 
numerical examples, and Section~\ref{sec:concl} is the
conclusion.

\section{Preliminaries and notation}
\label{sec:prelim}

\subsection{General notation}
Given $A\in\R^{n\times{n}}$, let $\rho(A)$ denote its spectral radius and
let $A\geq 0$ mean that all its elements are non-negative. A
  nonnegative matrix $A$ is irreducible if for any $i,j\in\until{n}$, there
  exists a $k=k(i,j)\leq n-1$ such that the $ij$ entry of $A^k$ is
  positive. Alternatively, if $A\geq 0$ is regarded as a weighted adjacency
  matrix of some directed graph $\mathcal{G}$, $A$ is irreducible if and
  only if the graph $\mathcal{G}$ is strongly connected.
If $A\geq 0$ is
irreducible, then, by the Perron-Frobenius theorem~\cite[Theorem~8.4.4.]{RAH-CRJ:12}, its eigenvalue
with largest magnitude $\subscr{\lambda}{max}(A)$ is real, simple, and
equal to $\rho(A)>0$. This eigenvalue is called the
Perron-Frobenius or dominant eigenvalue and has associated left and
right Perron-Frobenius or dominant eigenvectors with positive entries
(normalized to have unit sum, by convention).

Let $\norm{\cdot}$ denote an arbitrary norm,
$\norm{\cdot}_p$ denote the $\ell_p$-norm, and
$\norm{\cdot}_{p,Q}:=\norm{Q\cdot}_p$ with $Q$ being a positive definite
matrix denote a weighted $\ell_p$-norm.  When the argument of a norm is a
matrix, we refer to its respective induced matrix norm. Given two vectors
$x,y\in\R^n$, we denote $x\ll y$ when $x_i<y_i$ for every $i$; $x\leq y$ when
$x_i\leq y_i$ for every $i$; and $x<y$ when $x\leq y$ and $x\neq y$.

Let $I_n$ be the $n\times n$ identity matrix, $\vect{1}_n,\vect{0}_n\in\R^n$
be the all-ones and all-zeros vector with $n$ entries
respectively. Let $\vect{0}_{n\times{n}}$ be the $n\times n$ zero matrix. 
Let
$\diag(X_1,\dots,X_N)\in\R^{\sum^N_{i=1}n_i\times\sum^N_{i=1}n_i}$
represent a block-diagonal matrix whose elements are the matrices
$X_1\in\R^{n_1\times n_1},\dots,X_N\in\R^{n_N\times n_N}$. Given a 
vector $x\in\R^n$, $\diag(x)=\diag(x_1,\cdots,x_n)$. 
Let $\R_{\geq
  0}$ be the set of non-negative real numbers.  Given $x_i\in\R^{k_i}$, for
$i\in\until{N}$, we let $(x_1,\dots,x_N)=\begin{bmatrix}x_1^\top
&\dots&x_N^\top\end{bmatrix}$.

Finally, we recall a classic monotonicity property. If $A$ and $A'$ are
square matrices of the same dimension,
\begin{equation} \label{eq:monotonicity}
  0 \leq A \leq A' \quad \implies \quad \rho(A)\leq\rho(A'),
\end{equation}
where $ A \leq A'$ means $0\leq A'-A$.
\subsection{Matrix measures}
Given $A\in\R^{n\times{n}}$ and norm $\norm{\cdot}$ on
$\real^n$, its associated 
matrix measure is 
$\mu(A) = \lim_{h\to 0^{+}}\frac{\|I_n+hA\|-1}{h}$~\cite{MV:02,SC:19}.
Given $x\in\real^n$ and $\xi\gg\vect{0}_n$, the 
weighted $\ell_\infty$-norm is 
$\norm{x}_{\infty,\diag(\xi)}=\norm{\diag(\xi)x}_{\infty}$ and its associated
matrix measure is
%
\begin{align*}
  \mu_{\infty,\diag(\xi)}(A) &= \!\max_{i\in\until{n}} \!\Big( a_{ii} + \xi_i
  \sum\nolimits_{j=1,j\neq i}^n |a_{ij}| /\xi_j\Big).
\end{align*}
Given a Metzler matrix $M\in\real^{n\times{n}}$ and a scalar $b$, 
\begin{equation}
\label{eq:1-aux}
\begin{aligned} 
  M \xi \leq b \xi &\enspace\iff\enspace 
  \mu_{\infty,\diag(\xi)^{-1}}(M) \leq b .
\end{aligned}
\end{equation}

\section{Exponential convergence and matrix measures}
\label{sec:exp-mmeasures}
The following result combines the matrix measure results shown above with
the Coppel's inequality as stated in~\cite[Theorem~22, (Chapter~2,
  page~52)]{MV:02}. To the best of our knowledge, this connection and the
result in~\cite{MV:02} have not been explicitly exploited before. This result
will be useful for the paper's main theorem.

\begin{theorem}[Exponential convergence from Coppel's inequality]
  \label{theorem:append}
  Consider a smooth dynamical system $\dot{x}=f(x)$ with a convex compact
  invariant set $\mathcal{X}$ and an equilibrium point
  $x^*\in\mathcal{X}$. Write the system as
  \begin{equation}
    \label{eq:f-D}
    \dot{x}=\mathcal{D}(x,x^*)(x-x^*).
  \end{equation}
  where $\mathcal{D}(x,x^*)\in\R^{n\times n}$ is a function of $x$ and $x^*$.
  Let $\norm{\cdot}$ be a norm and $\mu$ be its associated matrix measure
  $\mu$.  If $\mu(\mathcal{D}(x,x^*))\leq-c$ for any $x\in\mathcal{X}$ and some $c>0$, then $x^*$ is
  the unique exponentially stable equilibrium point in $\mathcal{X}$ and
  exponential convergence 
  with at least rate $c$. Moreover,
  $V(x)=\norm{x-x^*}$ is a global Lyapunov function for $x^*$ in
  $\mathcal{X}$.
\end{theorem}
\begin{proof}
  First, 
  it is always possible~\cite[Lemma~17, Chapter~2,
    page~52]{MV:02} to write $f$ in the form~\eqref{eq:f-D} using the
  fundamental theorem of calculus and the convexity of $\mathcal{X}$.
  Second, as argued in~\cite[Chapter~1, page~3]{WAC:1965}, since the
  right-hand derivative of $x(t)-x^*$ is $\dot{x}(t)$ at any $t\geq 0$, the right-hand derivative $\frac{d^+}{dt}\norm{x-x^*}$ exists and moreover
  \begin{align*}
    \frac{d^+}{dt}\norm{x-x^*}
    &=\lim_{h\to 0^+}\frac{\norm{x-x^*+h\dot{x}}-\norm{x-x^*}}{h}\\
    &=\lim_{h\to 0^+}\frac{\norm{I_n+h\mathcal{D}(x,x^*)}-1}{h}\norm{x-x^*}\\
    &\leq\mu(\mathcal{D}(x,x^*))\norm{x-x^*}\leq -c\norm{x-x^*},
  \end{align*}  
 where the second inequality follows from Coppel's inequality as in~\cite[Theorem~3,
    Chapter~3]{WAC:1965} and in~\cite[Theorem~22, Chapter~2,
    page~52]{MV:02}, and the third inequality follows from  
  the negative matrix measure assumption.  Therefore,
  applying Gr\"{o}nwall's inequality, any trajectory $x(t)$ starting in
  $\mathcal{X}$ satisfies $\norm{x(t)-x^*}\leq
  e^{-ct}\norm{x(0)-x^*}$. Moreover, $x^*$ is the unique globally
  exponentially stable equilibrium in $\mathcal{X}$.
  
  Finally, observe that $V(x)=\norm{x-x^*}$, $x\in\mathcal{X}$, is a
  Lyapunov function with respect to $x^*$ since (i) it is globally proper,
  i.e., for each $\ell>0$, the set $\setdef{x\in\mathcal{X}}{V(x)\leq\ell}$
  is compact (since $\mathcal{X}$ is compact), (ii) it is positive definite
  on $\mathcal{X}$, (iii) strictly decreasing for any $x\neq x^*$ on
  $\mathcal{X}$.  This finishes the proof.
\end{proof}
  
\section{The Simplicial SIS model}
\label{sec:simplicial}
We study the following multi-group deterministic model, which can be regarded as 
a mean-field approximation of a more realistic stochastic model --- mean-field models are used because their dynamics are deterministic (described by ODEs) and their states describe multiple large populations of individuals.

\begin{figure}[ht]
  \centering
  \includegraphics[width=\linewidth]{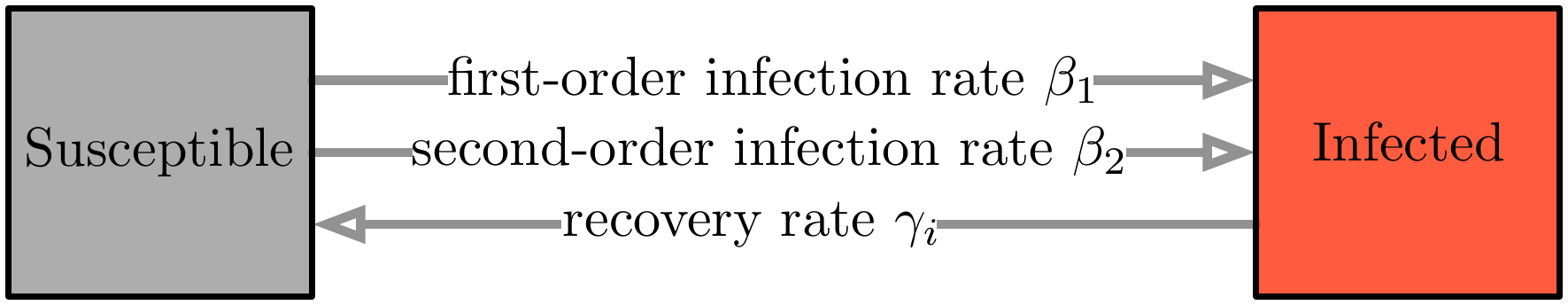}
  \caption{Simplicial SIS as a compartmental model}\label{fig:compartmental-fig}
\end{figure}

\begin{definition}[The simplicial SIS model]
\label{def:simplicial_model}
Assume $x\in[0,1]^n$, and let $\beta_1,\beta_2>0$ and $\gamma_i>0$,
$i\in\until{n}$. Then, the \emph{simplicial SIS model} is, for any
$i\in\until{n}$,
\begin{multline}
  \label{model_i}
  \dot{x}_i=-\gamma_i x_i+\beta_1(1-x_i)\sum\nolimits_{j=1}^n a_{ij}x_j\\
  +\beta_2(1-x_i)\sum\nolimits_{j,k=1}^n b_{ijk}x_jx_k, 
\end{multline}
or, in its matrix form, with $\Gamma=\diag(\gamma_1,\cdots,\gamma_n)$, the model is 
\begin{equation}
\label{model_matrix}
\begin{aligned}
\dot{x}&=-\Gamma x+\beta_1(I_n-\diag(x))Ax\\
&\quad+\beta_2(I_n-\diag(x))(x^\top B_1x,\cdots,x^\top B_nx)^\top
\end{aligned}
\end{equation}
where $B_i=\begin{bmatrix}
b_{i11}&\cdots&b_{i1n}\\
\vdots&&\vdots\\
b_{in1}&\cdots&b_{inn}
\end{bmatrix}$, $i\in\until{n}$, and $A=(a_{ij})$ are arbitrary nonnegative matrices.
\end{definition}

We now provide some remarks about this definition.

\begin{remark}[Interpretation of Definition~\ref{def:simplicial_model}]
\label{rem:1}
(i) 
Matrix $A\geq 0$ represents the pairwise contact rate
    between the agents, interpreted as a weighted adjacency matrix: $a_{ij}>0$ if agent $i$ (i.e., population or individual) is in contact with $j$, i.e., there is a directed edge from $i$ to $j$ with weight $a_{ij}$; and
    the magnitude of $a_{ij}$ indicates the contact frequency: the larger,
    the more positive effect on the infection spread. Now, for 
     matrix $B_i\geq 0$, $b_{ijk}>0$ if agent $i$ can have a
    simultaneous interaction with $j$ and $k$, and the magnitude of
    $b_{ijk}$ indicates the strength of the interaction; i.e., there is a hyperedge $(i,j,k)$ with weight $b_{ijk}$. Thus, the elements
    of $B_i$ indicate higher-order interaction effects that two agents
    jointly have over $i$. This is a key structural difference with the
    classical multi-group SIS model, see
    Figure~\ref{fig:comparison-models}.  Finally, $a_{ii}>0$ and
    $b_{iii}>0$ indicate different orders on the effect of actions taken by
    $i$ that increase the effect of the infection, and $b_{ijj}>0$
    indicates the higher-order effects of $j$'s actions over $i$.

(ii) If 
our model 
is 
strictly defined over a simplicial
  complex, then $A$ and $B_i$ should be symmetric and have 
  joint
  restrictions on 
  their elements. However, in our work, we do
    not restrict $A$ or $B_i$ to be symmetric and consider a more general
  mathematical model. We keep the term \emph{simplicial} in the title of
  the model since the special case of simplicial complexes inspired the
  more general model.
%
 
(iii) The parameter $\gamma_{i}$ is the recovery rate of agent $i$ 
  from the infection.  Parameters $\beta_1$ and
  $\beta_2$ are the infection rates at which an agent may get infected due
  to pairwise or higher-order interactions
  respectively. Figure~\ref{fig:compartmental-fig} shows how these
  parameters modulate the proportion of infected and susceptible people
  inside a population, or equivalently, the changes in the probability for
  an individual to be infected or susceptible.
\end{remark}

We revisit the qualitative behavioral domains that a multi-group SIS model
with higher-order terms must display. 

\begin{definition}[Epidemic domains] 
\label{def:epidemic_domains}
Consider the simplicial SIS model with fixed parameters $\Gamma$, $A$ and
$B_i$ for all $i\in\until{n}$.  According to the values of parameters
$(\beta_1,\beta_2)$, the system is in the:
\begin{enumerate}
\item \emph{Disease-free domain:} the disease-free equilibrium $\vect{0}_n$
  is the unique equilibrium and globally stable.
\item \emph{Bistable domain:} the disease-free equilibrium is locally
  asymptotically stable and there exists an endemic equilibrium
  $x^*\gg\vect{0}_n$ which is also locally asymptotically stable.
\item \emph{Endemic domain:} the disease-free equilibrium is unstable and
  there exists a unique endemic equilibrium that is asymptotically stable
  in $[0,1]^n\setminus\{\vect{0}_n\}$.
\end{enumerate}
\end{definition}

The following 
theorem 
describes the behavior of the scalar
version 
of the simplicial SIS model 
in~\cite{II-GP-AB-VL:19}; although~\cite{II-GP-AB-VL:19} does not state its
results 
as a theorem, 
we present them as such for 
comparison purposes.

\begin{theorem}[Dynamics of the scalar model in~\cite{II-GP-AB-VL:19}]
  \label{th:scalar}
  Consider the \emph{scalar simplicial SIS model}
  \begin{equation}
    \label{eq:scalar_model}
    \dot{y}=-\gamma y + \beta_1(1-y)y + \beta_2(1-y)y^2
  \end{equation}
  with $y\in[0,1]$ and $\gamma,\beta_1,\beta_2>0$. 
  Then,
  the set $[0,1]$ is invariant and $0$ is an equilibrium point. 
  Define
  $v_c(\beta_2/\gamma)=2\sqrt{\frac{\beta_2}{\gamma}}-\frac{\beta_2}{\gamma}$
  and the two variables $\nu_\pm=\frac{1}{2}(1-\frac{\beta_1}{\beta_2})
  \pm\frac{1}{2\beta_2}\sqrt{(\beta_1-\beta_2)^2-4\beta_2(\gamma-\beta_1)}$. Moreover,
\begin{description}
\item[Disease-free domain:] If either $\frac{\beta_2}{\gamma}\leq 1$ and $\frac{\beta_1}{\gamma}\leq1$, or $\frac{\beta_2}{\gamma}> 1$ and $\frac{\beta_1}{\gamma}<v_c(\beta_2/\gamma)$, then 
\begin{enumerate}
\item\label{t-1} $0$ is the unique equilibrium point in $[0,1]$,
\item $0$ is globally asymptotically stable in $[0,1]$.
  \setcounter{saveenum}{\value{enumi}}
\end{enumerate}
\item[Bistable domain:] If $\frac{\beta_2}{\gamma}>1$ and $v_c(\beta_2/\gamma)<\frac{\beta_1}{\gamma}<1$, then $\nu_-,\nu_+\in(0,1]$ and
\begin{enumerate} \setcounter{enumi}{\value{saveenum}}
\item\label{t-2} $0$ is locally asymptotically stable in $[0,\nu_-)$,
\item $\nu_+$ is a locally asymptotically stable equilibrium in $(\nu_-,1]$, and
\item $\nu_-$ is an unstable equilibrium.
    \setcounter{saveenum}{\value{enumi}}
\end{enumerate}
\item[Endemic domain:] If $\frac{\beta_1}{\gamma}>1$, then 
\begin{enumerate} \setcounter{enumi}{\value{saveenum}}
\item\label{t-3} $0$ is unstable,
\item $\nu_+$ is the unique equilibrium in $(0,1]$ and is globally asymptotically stable in $(0,1]$.
%
\end{enumerate}
\end{description}
\end{theorem}

\begin{figure}[ht]
  \centering
  \includegraphics[width=.8\linewidth]{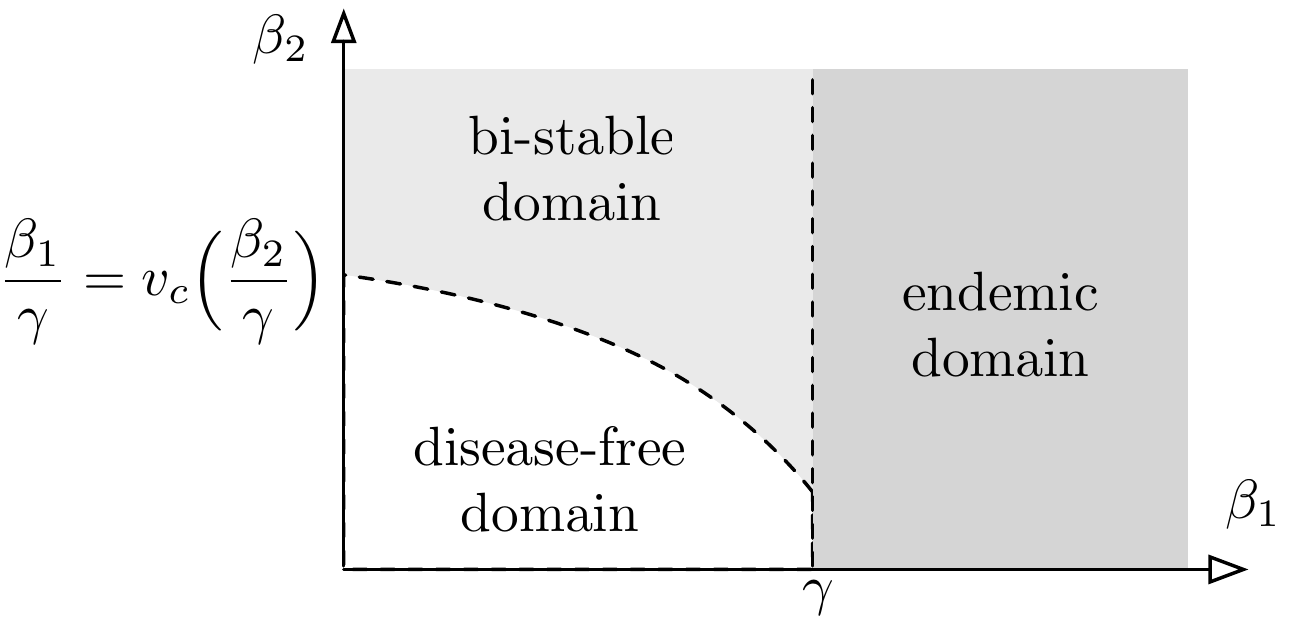}
  \caption{Epidemic diagram for the scalar simplicial SIS model (see
    Theorem~\ref{th:scalar}).}
\end{figure}

Notice the polynomial resemblance of the scalar model
in~\eqref{eq:scalar_model} and our multi-group simplicial model
in~\eqref{model_matrix}.

\section{Analysis of the model}
\label{sec:analysis}
First, we establish properties of the model independently from their parameter values.

\begin{lemma}[General properties of the simplicial SIS model]
\label{lem:general}
Consider the simplicial SIS model with an irreducible $A\ge0$ and arbitrary
$B_i\ge0$. Then,
\begin{enumerate}
\item\label{l-1} The set $[0,1]^n$ is an invariant set.
\item\label{l-2} If $x(0)>\vect{0}_n$, then $x(t)\gg\vect{0}_n$ for any $t>0$.
\item\label{l-3} The origin $\vect{0}_n$ is an equilibrium of the system and there are no other equilibria on the boundary of the set $[0,1]^n$.
\end{enumerate}
\end{lemma}
\begin{proof}
Let $f(x)$ be the right-hand side of equation~\eqref{model_matrix}. 
We first prove statement~\ref{l-1}. Following Nagumo's
theorem~\cite[Theorem~4.7]{FB-SM:15} we analyze the vector field at the
boundary of 
$[0,1]^n$. From equation~\eqref{model_i}, we see that
1) $f_i(x)\geq 0$ for all $x\in[0,1]^n$ such that $x_i=0$ for some
$i\in\until{n}$; 2) $f_i(x)< 0$ for all $x\in[0,1]^n$ such that $x_i=1$
for some $i\in\until{n}$; from which it follows that $[0,1]^n$ is an
invariant set. This proves statement~\ref{l-1}.

Set the change of variables $y=e^{\Gamma t}x$. Then, from equation~\eqref{model_matrix},
\begin{equation}
\label{eq:aux_y}
\begin{aligned}
  \dot{y}&=\diag(e^{\gamma_1t},\cdots,e^{\gamma_nt})(I_n-\diag(x))(\beta_1Ax\\
  &\quad+\beta_2(x^\top B_1 x,\cdots,x^\top B_n x)^\top).
\end{aligned}
\end{equation}
Since $x(0)\in[0,1]^n$, notice that $\dot{y}(t)\geq\vect{0}_n$ for any
$t\geq 0$, and so there is the monotonicity property $y(t_1)\geq y(t_0)$
for any $t_1,t_0\geq 0$. Now, we prove statement~\ref{l-2} by
contradiction. Let us assume that $x(0)>\vect{0}_n$, which implies 
$y(0)>\vect{0}_n$, and that there exists some $i\in\until{n}$ and $T>0$ such
that $y_i(T)=0$. Then, because of the monotonicity property, $y_i(t)=0$ for
all $t\in[0,T]$, which implies that $x_i(t)=0$. Then, from the equilibrium
equation of~\eqref{eq:aux_y}, we have that
$0=\beta_1e^{\gamma_it}\sum^n_{\substack{j=1\\j\neq
    i}}a_{ij}e^{-\gamma_jt}y_j(t)+\beta_2e^{\gamma_i
  t}\sum^n_{\substack{j=1\\j\neq i}}\sum^n_{\substack{k=1\\k\neq
    i}}b_{ijk}e^{-\gamma_jt}e^{-\gamma_kt}y_j(t)y_k(t)$ for any
$t\in[0,T]$, and since all terms are non-negative, it follows that
$y_j(t)=0$ for $t\in[0,T]$ and all $j$ such that $a_{ij}> 0$. Then, 
for any such $j$, we repeat the same analysis we just did and find that $y_k(T)=0$ for all
$t\in[0,T]$ and all $k$ such that $a_{jk}>0$. Then,
since
$A$ is irreducible, we could continue 
repeating this procedure and finally obtain
$y(t)=\vect{0}_n$ for all $t\in[0,T]$. This gives a contradiction, since we
had that $y(0)>\vect{0}_n$ because of $x(0)>\vect{0}_n$. 
Then, 
$y(t)\gg\vect{0}_n$
implies 
$x(t)\gg\vect{0}_n$ for $t>0$ and finish the
proof of statement~\ref{l-2}.
  
Finally, we prove statement~\ref{l-3}. First, let us introduce the functions $h_+(z)=\frac{z}{1+z}$ for any $z\in\R_{\geq 0}$ and $h_-(z)=\frac{z}{1-z}$ for $z\in[0,1)^n$. We also introduce $H_+(y)=(h_+(y_1),\dots,h_+(y_n))^\top$ for $y\geq\vect{0}_n$, and $H_-(y)=(h_-(y_1),\dots,h_-(y_n))^\top$ for $y\in[0,1)^n$. 

Now, it is immediate from equation~\eqref{model_matrix} that $\vect{0}_n$ is an equilibrium point, and observe that 
there is no 
equilibrium point $x^*$ such that $x^*_i=1$ for some $i\in\until{n}$, since that would imply that $(f(x^*))_i<0$. 
Now, assume $x^*$ is an equilibrium point such that $x^*_i=0$ for some $i\in\until{n}$. Let 
$B_{x^*}:=({x^*}^\top B_1 x^*,\cdots,{x^*}^\top B_n x^*)^\top$. First, from the equilibrium equation of the system~\eqref{model_matrix}, since $x^*\ll\vect{1}_n$, we obtain
\begin{align*}
&\vect{0}_n=-\Gamma x^*+(I_n-\diag(x^*))(\beta_1 Ax^*+\beta_2B_{x^*})\\
&\Longleftrightarrow  (I_n-\diag(x^*))^{-1}x^*=\Gamma^{-1}(\beta_1Ax^*+\beta_2B_{x^*})\\
&\Longleftrightarrow H_-(x^*)=\Gamma^{-1}(\beta_1Ax^*+\beta_2B_{x^*})\\
&\Longleftrightarrow H_+(\Gamma^{-1}(\beta_1Ax^*+\beta_2B_{x^*}))=x^*,
\end{align*}  
and so $x_i^*=h_+(\frac{\beta_1}{\gamma_i}\sum^n_{j=1}a_{ij}x_j^*+\frac{\beta_2}{\gamma_i}{x^*}^{\top}B_ix^*)$. Then, since $x^*_i=0$, this implies that $x_j^*=0$ for all $j$ such that $a_{ij}>0$. Then, since $A$ is irreducible, we could iterate 
this procedure and conclude that $x^*=\vect{0}_n$. 
Therefore, 
any equilibrium point at the boundary of $[0,1]^n$ must be the origin. This 
proves 
statement~\ref{l-3}. 
\end{proof}

From an epidemiological perspective, Lemma~\ref{lem:general} shows two important things for the well-posedness of the simplicial SIS model: 1) each entry of the state vector of the model can represent a proportion or probability; 2) there cannot exist another type of equilibria than disease-free or endemic ones. Now we present our main result.

\begin{theorem}[The simplicial SIS model and its different epidemiological domains]
\label{th:simplicial_main}
Consider the simplicial SIS model with an irreducible $A\ge0$ and arbitrary
$B_i\ge0$. Define $\ybar\in\{0,1\}^n$ by $(\ybar)_i=1$ if
$B_i\neq\vect{0}_{n\times{n}}$ and $(\ybar)_i=0$ otherwise.
\begin{description}
\item[Disease-free domain:] If 
  \begin{equation*}
    \rho(\beta_1\Gamma^{-1}A+\beta_2\Gamma^{-1}(\vect{1}_n^\top
  B_1,\cdots,\vect{1}_n^\top B_n)^\top)<1,
  \end{equation*}
  then
  \begin{enumerate}
\item\label{fact-i} $\vect{0}_n$ is the unique equilibrium point in $[0,1]^n$,
\item\label{fact-ii} $\vect{0}_n$ is globally exponentially stable in $[0,1]^n$ with
  Lyapunov function $V(x)=\norm{x}_{1,\diag(v)\Gamma^{-1}}=v^\top\Gamma^{-1}x$, where $v$ is the
  dominant left eigenvector of
  $\beta_1\Gamma^{-1}A+\beta_2\Gamma^{-1}(\vect{1}_n^\top
  B_1,\cdots,\vect{1}_n^\top B_n)^\top$.
\setcounter{saveenum}{\value{enumi}}  
\end{enumerate}
\item[Bistable domain:] If $\beta_1\rho(\Gamma^{-1}A)<1$ and
  \begin{equation*}
    \min_{i \text{ s.t. }B_i\neq\vect{0}_{n\times n}}
    \Big(\frac{\beta_1}{\gamma_i}(A\ybar)_i  + \frac{\beta_2}{2\gamma_i}
    \ybar^\top B_i\ybar\Big)\geq 2,
  \end{equation*}
  then
\begin{enumerate}\setcounter{enumi}{\value{saveenum}}
\item\label{fact-iii} $\vect{0}_n$ is a locally exponentially stable equilibrium,
\item\label{fact-iv} there exists an equilibrium point $x^*\gg\vect{0}_n$ such that
  $x^*_i\geq\frac{1}{2}$ for any $i$ such that
  $B_i\neq\vect{0}_{n\times{n}}$, and
\item\label{fact-v} any such equilibrium point $x^*$ is locally exponentially stable.
    \setcounter{saveenum}{\value{enumi}}
\end{enumerate}
\item[Endemic domain:] If $\beta_1\rho(\Gamma^{-1}A)>1$, then
\begin{enumerate}\setcounter{enumi}{\value{saveenum}}
\item\label{fact-vi} $\vect{0}_n$ is an unstable equilibrium,
\item\label{fact-vii} there exists an equilibrium point $x^*\gg \vect{0}_n$ in $[0,1]^n$, and
\item\label{fact-viii} if $\beta_2$ is sufficiently small, then $x^*$ is unique in $(0,1]^n$
  and it is globally exponentially stable in
  $[0,1]^n\setminus\{\vect{0}_n\}$, with Lyapunov function
    $V(x)=\norm{x-x^*}_{\infty,\diag(x^*)^{-1}}$, $x\in\mathcal{X}$.
\end{enumerate}
\end{description}
Moreover, if $\beta_1\rho(\Gamma^{-1}A)<1$, then the system is either in the disease-free domain or in the bi-stable domain.
\end{theorem}
%
%
\begin{remark}[About Theorem~\ref{th:simplicial_main}] 
  \label{rem-main-th} 
(i) Pick $\beta_1$ satisfying $\beta_1\rho(\Gamma^{-1}A)<1$. Assume
    either that each $B_i$ is non-zero (all agents have higher-order interactions), or that each non-zero $B_i$ has a positive 
    $i$th diagonal entry (agent $i$ suffers from cumulative infection effects from her neighbors; e.g., see Remark~\ref{rem:1}). Then there exists some 
    $\hat{\beta}_2>0$ such that the second condition for the bistable domain is satisfied for $\beta_2=\hat{\beta}_2$ and the
    simplicial SIS model is in the bistable domain 
    for any
    $\beta_2\geq\hat{\beta}_2$.
%

(ii) Compared to the scalar model in Theorem~\ref{th:scalar}, the
    sufficient conditions in Theorem~\ref{th:simplicial_main} defining the
    different domains for the simplicial SIS model do not exhaust all the
    possible values for $(\beta_1,\beta_2)$; e.g., see~Fig.~\ref{f:sim1}. Despite this gap, our theorem
    rigorously establishes the following crucial qualitative behavior:
    assume there exist parameters $(\beta_1,\beta_2)$ that satisfy the
    sufficient condition for the bistable region in
    Theorem~\ref{th:simplicial_main}, then we can show the system can
    transition from the disease-free domain to the bistable domain (and
    vice versa) by modifying $\beta_2$.  This transition, presented as a
    novelty for the scalar model, is also a novelty of the simplicial SIS
    model.

(iii) In the literature on the classical multi-group SIS model, where
    only the disease-free and endemic domains exist, the number
    $\beta_1\rho(\Gamma^{-1}A)$ is known as the \emph{reproduction number}
    and its value has been used to determine whether the system is in the
    endemic domain or not. This 
    number has a similar role for
    the simplicial SIS model. 
    Indeed, if all higher-order interaction matrices $B_i$ are equal to
    zero, then our theorem reduces to and restates some properties of the
    classical multi-group SIS model, e.g., see~\cite[Theorems~4.2
      and~4.3]{WM-SM-SZ-FB:16f}.

(iv) In the classical SIS multi-group model, the
      work~\cite{AF-AI-GS-JJT:07} uses the Lyapunov function
      $V(x)=\norm{x-x^*}_{1,\diag(x^*)}$ to show asymptotic convergence to
      the a unique endemic state $x^*\in[0,1]^n\setminus\{0\}_n$. Note that 
      Theorem~\ref{theorem:append} generalizes~\cite[Theorem~2.7]{AF-AI-GS-JJT:07}.
\end{remark}
%
\begin{proof}[Proof of Theorem~\ref{th:simplicial_main}]
Let us consider the functions $H_+$ and $h_+$ introduced in the proof of Lemma~\ref{lem:general}. Let $\bar{A}:=\beta_1\Gamma^{-1}A$, $\bar{B}_i:=\frac{\beta_2}{\gamma_i}B_i$ for $i\in\until{n}$, and let $\dot{x}:=f(x)$. We introduce the following result: if $\vect{0}_n\leq y\ll z$ and $C\geq 0$ an $n\times{n}$ irreducible matrix, 
then $H_+(Cy)\ll H_+(Cz)$. This follows from the fact that, since
$C$ is irreducible, there exists at least one positive entry in some
off-diagonal entry in any row of $C$, and so $C(z-y)\gg \vect{0}_n$. Then,
$Cz\gg Cy$, and since $h_+$ is monotonically increasing, then
$H_+(Cy)\ll H_+(Cz)$. Similarly, if $\vect{0}_n\leq y\leq z$ and $C\geq 0$ an $n\times{n}$ matrix (not necessarily irreducible), then $H_+(Cy)\leq H_+(Cz)$. 
We use these results throughout the rest of this proof. 

We first prove fact~\ref{fact-i} by contradiction. 
Let $x^*$ be an equilibrium different than the origin. From the proof of Lemma~\ref{lem:general}, 
$x^*$
is an equilibrium point if and only if
$H_+(\bar{A}x^*+({x^*}^\top\bar{B}_1x^*,\cdots,{x^*}^\top\bar{B}_nx^*)^\top)=x^*$,
i.e., if and only if $x^*$ is the fixed point of the map
$H(x):=H_+(\bar{A}x+(x^\top\bar{B}_1x,\cdots,x^\top\bar{B}_nx)^\top)$. Now,
  observe that
\begin{align*}
H(x^*)&\leq \bar{A}x^* + ({x^*}^\top\bar{B}_1x^*,\cdots,{x^*}^\top\bar{B}_nx^*)^\top\\
&\leq \bar{A}x^*+(\vect{1}_n^\top\bar{B}_1x^*,\cdots,\vect{1}_n^\top\bar{B}_nx^*)^\top
\end{align*}
where the first inequality follows from $h_+(z)\leq z$ for $z\in(0,1]$ and the second one from $x^*\in[0,1]^n$. Now, observe that if $\vect{0}_n\leq x\leq y$ then $\vect{0}_n\leq H(x)\leq \bar{A}x+(\vect{1}_n^\top\bar{B}_1x,\cdots,\vect{1}_n^\top\bar{B}_nx)^\top\leq \bar{A}y+(\vect{1}_n^\top\bar{B}_1y,\cdots,\vect{1}_n^\top\bar{B}_ny)^\top$; and so, the $k$th iteration of the map $H$ satisfies: $\vect{0}_n\leq H^k(x^*)\leq (\bar{A}+(\vect{1}_n^\top\bar{B}_1,\cdots,\vect{1}_n^\top\bar{B}_n)^\top)^kx^*$. Now, assume by contradiction that $x^*\neq\vect{0}_n$. Then, from our previous calculations, $0\leq \norm{H^k(x^*)-H^k(0)}\leq \norm{(\bar{A}+(\vect{1}_n^\top\bar{B}_1,\cdots,\vect{1}_n^\top\bar{B}_n)^\top)^k}\norm{x^*}$ since $H^k(\vect{0}_n)=\vect{0}_n$ and where the last inequality follows from the definition of induced norms. Now, by hypothesis, we have that $\rho(\bar{A}+(\vect{1}_n^\top\bar{B}_1,\cdots,\vect{1}_n^\top\bar{B}_n)^\top)<1$, and so, it follows that $\lim_{k\to\infty}(\bar{A}+(\vect{1}_n^\top\bar{B}_1,\cdots,\vect{1}_n^\top\bar{B}_n)^\top)^k=\vect{0}_{n\times{n}}$. Then, by the Sandwich theorem, $\lim_{k\to\infty}\norm{H^k(x^*)-H^k(0)}=0$ but recalling that $H^k(x^*)=x^*$ since $x^*$ is a fixed point of $H$, we obtain $\norm{x^*}=\vect{0}_n$, which is a contradiction. Then, $\vect{0}_n$ is the unique fixed point in $[0,1]^n$ for the map $H$, and thus, the unique equilibrium point for the system.

Now we prove fact~\ref{fact-ii}. 
Since $\bar{A}\geq 0$ is irreducible, let $v\gg\vect{0}_n$ be the left Perron-Frobenius eigenvector of 
$\bar{A}+(\vect{1}_n^\top \bar{B_1},\cdots,\vect{1}_n^\top \bar{B_n})^\top\geq 0$
~\cite[Theorem~8.4.4.]{RAH-CRJ:12}, and let $\lambda:=\rho(\bar{A}+(\vect{1}_n^\top \bar{B}_1,\cdots,\vect{1}_n^\top \bar{B}_n)^\top)$ be its 
eigenvalue. 
%
Set $y=v^\top\Gamma^{-1}x$, then $\dot{y}=v^\top\Gamma^{-1}\dot{x}$ and
\begin{align*}
\dot{y}&\leq-v^\top x+v^\top(\bar{A}x+(x^\top\bar{B}_1x,\cdots,x^\top\bar{B}_nx)^\top)\\
&\leq(-1+ \lambda)v^\top x\\
&=(-1+\lambda)v^\top\Gamma\Gamma^{-1}x\leq(-1+ \lambda)(\min_i\gamma_i)y,
%
\end{align*}
where the first inequality follows from $v^\top\Gamma^{-1}(I_n-\diag(x))\leq v^\top\Gamma^{-1}$ for any $x\in[0,1]^n$. Set $q:=(-1+ \lambda)(\min_i\gamma_i)<0$. Then, the Comparison Lemma~\cite{HKK:02} implies $0\leq y(t)\leq y(0)e^{qt}$ for $t\geq 0$; and thus $y(t)$ is a Lyapunov function. From this it follows that $x_i(t)\leq \frac{v^\top x(0)}{v_i}e^{qt}$ and so $\norm{x(t)}_1\leq C_o e^{qt}$ for some constant $C_o>0$, which finally implies that $\vect{0}_n$ is globally exponentially stable in $[0,1]^n$. 


We prove fact~\ref{fact-iii} by linearization. First, observe that the Jacobian
evaluated at the equilibrium point $\vect{0}_n$ is
$\dot{x}=(-\Gamma+\beta_1 A)x$. Since $A$ is irreducible, let
$v\gg\vect{0}_n$ be the right Perron-Frobenius eigenvector of
$\beta_1\Gamma^{-1}A$; and let $\rho$ denote its associated
eigenvalue. Note that $-\Gamma+\beta_1 A$ is Metzler and $(-\Gamma +\beta_1
A)v=(-1+\rho)\Gamma v\ll\vect{0}_n$ since $-1+\rho<0$ by assumption.
Using~\cite[Theorem 15.17]{FB:20}, we conclude that the matrix
$-\Gamma+\beta_1 A$ is Hurwitz and so the origin is locally exponentially
stable.

Now we prove fact~\ref{fact-iv} by finding a fixed point of map $H$ satisfying the stated conditions. 
First, we introduce the following result: for any $\alpha>1$, $h_+(\alpha
z)\geq z$ with $z\geq 0$ if and only if $z\leq 1-\frac{1}{\alpha}$. Now,
consider the vector $\ybar$ as in the theorem statement and define
$Y=\setdef{y\in[0,1]^n}{\frac{1}{2}\ybar\leq y \leq \vect{1}_n}$ and
$\theta:=\min_{i \text{ s.t. }B_i\neq\vect{0}_{n\times n}}
\Big(\frac{2\beta_1}{\gamma_i}(A\ybar)_i + \frac{\beta_2}{\gamma_i}
\ybar^\top B_i\ybar\Big)$.  Note that $\theta\geq 4$ by hypothesis. Let
$y\in Y$, then
\begin{equation}
  \label{eq:aux_H}
  \begin{aligned}
    H(y)&=H_+(\bar{A}y+(y^\top\bar{B}_1y,\cdots,y^\top\bar{B}_ny)^\top)\\
    &\geq H_+\Big(\frac{1}{2}\bar{A}\ybar+\frac{1}{4}(\ybar^\top\bar{B}_1,\cdots,\ybar^\top\bar{B}_n)^\top\ybar\Big),
  \end{aligned}
\end{equation} 
where the monotonicity of the function $h_+$ implies the inequality. 
Now, the $i$th entry of the argument of $H_+$ in right-hand side
of~\eqref{eq:aux_H} is $\frac{1}{4}(\frac{2\beta_1}{\gamma_i}\sum_{j=1}^na_{ij}(\ybar)_j+\frac{\beta_2}{\gamma_i}\ybar^\top
B_i\ybar)$. When $B_i\neq\vect{0}_{n\times{n}}$, we can lower bound the
$i$th entry by $\frac{1}{4}\theta$; and when $B_i=\vect{0}_{n\times{n}}$, 
by $0$. Therefore, from~\eqref{eq:aux_H},
$$H(y)\geq H_+\Big(\frac{1}{4}\theta\ybar\Big)\geq \frac{1}{2}\ybar$$ where the
last inequality follows from our statement at the beginning of the paragraph. Now,
from the fact that $h_+(z)\leq 1$ for any $z\geq 0$, then 
$H(y)=(H_+\bar{A}y+(y^\top\bar{B}_1y,\cdots,y^\top\bar{B}_ny)^\top)\leq\vect{1}_n$. Then,
we conclude that $H:Y\to Y$, and so $H$ is a continuous map that maps $Y$
into itself. The Brouwer Fixed-Point Theorem (e.g.,
see~\cite[Theorem~4.5]{JHS:16}) implies that there exists $y^*\in Y$ such
that $H(y^*)=y^*$, i.e., an equilibrium point $y^*$ for the system which
belongs to $Y$. 
This equilibrium point $y^*$ is not guaranteed to be unique. 
Moreover, from statement \ref{l-3} of
Lemma~\ref{lem:general}, we conclude that no entry of $y^*$ can be zero,
and so $y^*\gg\vect{0}_n$.

Now, we prove fact~\ref{fact-v} by linerarization. Let $x^*$ be an equilibrium of the system 
such that $x^*\geq \frac{1}{2}\ybar$ with $x^*\gg\vect{0}_n$. We denote by $Df$ the Jacobian of the vector field $f$. Then 
%
we obtain
\begin{equation*}
Df(x^*)=-\Gamma+\beta_1(I_n-\diag(x^*))A-\beta_1\diag(Ax^*)\\
+\beta_2(I_n-\diag(x^*))
O_1(x^*)
-\beta_2 O_2(x^*),
\end{equation*}
with $O_1(x^*):=({x^*}^\top(B_1+B_1^\top),\cdots,{x^*}^\top(B_n+B_n^\top))^\top$ and
$O_2(x^*):=\diag({x^*}^\top B_1 x^*,\cdots,{x^*}^\top B_n x^*)^\top$.  
Clearly, $Df(x^*)$ is a Metzler matrix. Now, observe that
\begin{multline}
\label{eq:aux_Jacob}
Df(x^*)x^*=-\beta_1\diag(Ax^*)x^*+\beta_2(I_n-\diag(x^*))({x^*}^\top B_1x^*,\cdots,{x^*}^\top B_n x^*)^\top\\
\quad-\beta_2\diag({x^*}^\top B_1x^*,\cdots,{x^*}^\top B_n x^*)x^*,
\end{multline}
where we simplified terms by using the equilibrium equation for the system~\eqref{model_matrix}. Let $(Df(x^*)x^*)_i$ be the $i$th entry of the left-hand side of equation~\eqref{eq:aux_Jacob}. Then
\begin{equation}
\label{eq:aux_Jacob_ith}
(Df(x^*)x^*)_i=-\beta_1\left(\sum_{i=1}^n a_{ij}x^*_j\right)x^*_i+\beta_2(1-2x_i^*)({x^*}^\top B_i x^*).
\end{equation}
First, consider $B_i\neq\vect{0}_{n\times{n}}$. Then, it follows that $x^*_i\geq\frac{1}{2}$ and that $(1-2x^*_i)\leq 0$. In turn we obtain, in~\eqref{eq:aux_Jacob_ith},
\begin{equation*}
\begin{aligned}
(Df(x^*)x^*)_i
\leq -\left(\beta_1\min_j
\left(\sum_{i=1}^n a_{ij}x^*_j\right)\right)x_i^*.
\end{aligned}
\end{equation*}
On the other hand, if $B_i=\vect{0}_{n\times{n}}$, then
$(Df(x^*)x^*)_i=-\beta_1\left(\sum_{i=1}^n a_{ij}x^*_j\right)x^*_i$
in~\eqref{eq:aux_Jacob_ith}.  Therefore, from these two cases, we conclude 
$Df(x^*)x^*\leq -dx^*$ for some $d>0$ since $A$ is irreducible. Then,
since $x^*\gg\vect{0}_n$ \cite[Theorem 15.17]{FB:20} implies that $Df(x^*)$
is Hurwitz, and so $x^*$ is locally exponentially stable. 

Now we prove fact~\ref{fact-vi}. First we prove that $\vect{0}_n$ is an
unstable equilibrium.  The linearization respect to the equilibrium point
$\vect{0}_n$ is $\dot{x}=(-\Gamma+\beta_1 A)x$. Let $v\gg\vect{0}_n$ be the
right Perron-Frobenius vector of the matrix $\beta_1\Gamma^{-1}A$, and let
$\rho$ be its associated eigenvalue. Now, since $-\Gamma+\beta_1
A$ is Metzler, $\rho>1$, and $A$ is irreducible; we
  invoke~\cite[E10.15]{FB:20} to conclude that the leading eigenvalue
of $-\Gamma+\beta_1 A$ is strictly positive.

Next we prove fact~\ref{fact-vii}. Define $Y=\setdef{y\in[0,1]^n}{c\leq y \leq \vect{1}_n}$ for a fixed $c=\alpha v$ 
and $0<\alpha<1$ small enough so that $c\leq \left(1-\frac{1}{\rho}\right)\vect{1}_n$, which is well-posed since $\rho>1$ by assumption. Let $y\in Y$, then 
\begin{align*}
H(y)=H_+(\bar{A}y+(y^\top\bar{B}_1y,\cdots,y^\top\bar{B}_ny)^\top)H_+(\bar{A}y)\geq H_+(\alpha\rho v)=H_+(\rho c)\geq c
\end{align*} 
where the inequalities are similar to the ones used in the 
the proof of 
fact~\ref{fact-iv}. 
Since we know also that $H(y)\leq\vect{1}_n$, the Brouwer
Fixed-Point Theorem 
implies that there exists
some $y^*\in Y$ such that $H(y^*)=y^*$, i.e., there exists an equilibrium point
$y^*\in Y$ for the system 
and, by construction,
$y^*\gg \vect{0}_n$.

Now, we prove fact~\ref{fact-viii}. First, we prove that $Y$ can be made a forward-invariant set for the system~\eqref{model_matrix}, and then we establish conditions for the existence of a unique exponentially stable equilibrium in $Y$. 
If $x\in Y$, then $x_i\in[c_i,1]$. Then, we can use Nagumo's theorem~\cite[Theorem~4.7]{FB-SM:15} and analyze the vector field at the
boundary of $Y$, which is an $n$-dimensional rectangle. As in the proof for statement~\ref{l-1} of Lemma~\ref{lem:general}, we have that $f_i(x)< 0$ for all $x\in Y$ such that $x_i=1$
for some $i\in\until{n}$. Then, we need to analyze only the case where $x\in Y$ with $x_i=c_i=\alpha v_i$ for some $i\in\until{n}$. Consider such $x$. Then, 
\begin{align*}
f_i(x)&=-\gamma_ic_i+\beta_1(1-c_i)\sum^n_{j=1}a_{ij}x_j+\beta(1-c_i)x^\top B_1 x\\
&\geq-\gamma_ic_i+\beta_1(1-c_i)\sum^n_{j=1}a_{ij}c_j\\
&=-\alpha\gamma_iv_i+\alpha\gamma_i(1-\alpha v_i)\frac{\beta_1}{\gamma_i}\sum^n_{j=1}a_{ij}v_j\\
&=\alpha\gamma_i(-1+(1-\alpha v_i)\rho(\bar{A}))v_i,
\end{align*}
and so $f_i(x)\geq 0$ if $\rho(\bar{A})\geq\frac{1}{1-\alpha v_i}$. Then, if $\rho(\bar{A})\geq\frac{1}{1-\alpha\max_i v_i}$, we conclude that 
$Y$ is forward invariant. 
Now, by construction of $Y$, 
we can make the parameter $\alpha>0$ arbitrarily small, and since $\rho(\bar{A})>1$ by assumption, then we conclude that $Y$ is forward invariant. 
Indeed, since $c\to\vect{0}_n$ as $\alpha\to 0$,
we can define the positively invariant set $Y$ to include any initial condition in $(0,1]^n$. Moreover, from statement~\ref{l-2} of Lemma~\ref{lem:general}, we conclude that any trajectory starting in $[0,1]^n\setminus\{\vect{0}_n\}$ eventually enters the positive invariant set $Y$. 

Now, let $x^*$ be an equilibrium point of the system belonging to $Y$, so
that $x^*\gg\vect{0}_n$ and let us consider the system~\eqref{model_matrix} starting in the set $Y$. By subtracting the right-hand side of the
equilibrium equation $\vect{0}_n=f(x^*)$, we can express the same
equation~\eqref{model_matrix} as
\begin{equation*}
\dot{x}=\Lambda(x,x^*)(x-x^*)+\beta_2\Omega(x,x^*)
\end{equation*}
with 
$\Lambda(x,x^*):=-\Gamma+\beta_1(I_n-\diag(x^*))A-\beta_1\diag(Ax)$ 
and
\begin{equation*}
\Omega(x,x^*):=(I_n-\diag(x))(x^\top B_1 x,\cdots,x^\top B_n x)^\top-(I_n-\diag(x^*))({x^*}^\top B_1 x^*,\cdots,{x^*}^\top B_n x^*)^\top,
\end{equation*}
and after some calculations, 
\begin{align*}
\Omega(x,x^*)= \Big( (I_n-\diag(x^*)) \begin{bmatrix}x^\top B_1^\top+{x^*}^\top B_1\\
\vdots\\
x^\top B_n^\top+{x^*}^\top B_n
\end{bmatrix}-\diag(x^\top B_1 x,\cdots,x^\top B_n x)\Big) (x-x^*).
\end{align*}
Then, we can have the alternative expression for~\eqref{model_matrix} as
\begin{equation*}
\dot{x}=\mathcal{D}(x,x^*)(x-x^*)
\end{equation*}
with $\mathcal{D}(x,x^*):=(\mathcal{D}_1(x,x^*)+\mathcal{D}_2(x,x^*))$ and 
\begin{align*}
\mathcal{D}_1(x,x^*)&:=-\Gamma+\beta_1(I_n-\diag(x^*))A+\beta_2(I_n-\diag(x^*))({x^*}^\top B_1,\cdots,{x^*}^\top B_n)^\top,\\
\mathcal{D}_2(x,x^*)&:=-\beta_1\diag(Ax)+\beta_2(I_n-\diag(x^*))(x^\top B_1^\top,\cdots,x^\top B_n^\top)^\top-\beta_2\diag(x^\top B_1x,\cdots,x^\top B_n x).
\end{align*}
Now, 
from the equilibrium equation $\vect{0}_n=f(x^*)$, we notice that $\mathcal{D}_1(x,x^*)x^*=\vect{0}_n$. 
Since $x\in Y$, notice that $-\diag(Ax)x^*\leq -\diag(Ac)x^*$ and $-\diag(x^\top B_1x,\cdots,x^\top B_n x)x^*\leq -\diag(c^\top B_1c,\cdots,c^\top B_n c)x^*\leq\vect{0}_n$. Using these results, we obtain
\begin{align*}
\mathcal{D}_2(x,x^*)x^*&\leq-\beta_1\diag(Ac)x^*+\beta_2(I-\diag(x^*))(x^\top B_1^\top x^*,\cdots,x^\top B_n^\top x^*)^\top\\
&\leq(-\beta_1\diag(Ac)+\beta_2(I-\diag(x^*))(\vect{1}_n^\top B_1^\top,\cdots,\vect{1}_n^\top B_n^\top)^\top)x^*.
\end{align*}
Now, since $A$ is irreducible and $c\gg\vect{0}_n$, for a fixed value of
$\beta_1>0$, there exists $\beta_2>0$ sufficiently small so that
$\mathcal{D}_2(x,x^*)x^*\leq -d x^*$ for some constant $d>0$.  Therefore, we
have shown that $\mathcal{D}(x,x^*)x^*\leq -dx^*$ for any $x\in Y$. Since
$\mathcal{D}(x,x^*)$ is Metzler (because both $\mathcal{D}_1(x,x^*)$ and
$\mathcal{D}_2(x,x^*)$ are Metzler) and $Y$ is a convex compact forward-invariant set,
we can use 
expression~\eqref{eq:1-aux} along with Theorem~\ref{theorem:append}. 
Then, we conclude
that $x^*$ is the unique globally exponentially stable equilibrium point in
$Y$, and, as a consequence of statement~\ref{l-3} from
Lemma~\ref{lem:general}, it has the same property over the set
$[0,1]^n\setminus\{\vect{0}_n\}$. This finishes the proof of fact~\ref{fact-viii}.

The last claim of the theorem follows from 
the proof of fact~\ref{fact-ii} 
which states that $\beta_1\rho(\Gamma^{-1}A)<1$ implies 
$\vect{0}_n$ is locally exponentially stable, and thus we are in either the disease-free or bistable domain. 
\end{proof}

\begin{theorem}[Algorithm for computing an endemic equilibrium]
  Consider the simplicial SIS model and assume that the system parameters
  satisfy the sufficient conditions in Theorem~\ref{th:simplicial_main} for
  the system to be in either the bistable or endemic domain. Define the
  map $\map{H_+}{\realnonnegative^n}{\realnonnegative^n}$ by
  $H_+(z)=(\frac{z_1}{1+z_1},\cdots,\frac{z_n}{1+z_n})^\top$ and
  $y_0\in(0,1)^n$ by
  \begin{equation*}
    y_{0} = \begin{cases} \frac{1}{2}\ybar, \quad &\text{ for the bistable domain,}\\
      (1-\frac{1}{\rho})u, \quad &\text{ for the endemic domain,}\end{cases}
  \end{equation*}
  with $(\rho,u)$ being the dominant right eigenpair of $\beta_1\Gamma^{-1}
  A$ and $\norm{u}_{\infty}=1$.  Then the sequence
  $(y_{k})_{k\in\natural}\subset(0,1)^n$ defined by
  \begin{align*}
    &y_{k+1}=H_+\big(\beta_1\Gamma^{-1}Ay_{k}+\beta_2\Gamma^{-1}(y_{k}^\top B_1
    y_{k},\cdots,y_{k}^\top B_n y_{k})^\top\big)
  \end{align*}
  is monotonic nondecreasing and $\lim_{k\to\infty}y_{k}=x^*$ is an endemic
  equilibrium (satisfying $y_0\ll x^* \ll \vect{1}_n$).
\end{theorem}
\begin{proof}
Let $f(x):=\beta_1\Gamma^{-1}Ax+\beta_2\Gamma^{-1}(x^\top B_1
x,\cdots,x^\top B_n x)^\top$ for $x\in[0,1]^n$. From the proof of
Theorem~\ref{th:simplicial_main}, 
there exists 
an endemic state
$x^*$ which satisfies $H_+(f(x^*))=x^*$. Now, we also know that
$H_+(f(y_{0}))\geq y_{0}$, and so $y_{1}\geq y_{0}$. Similarly, we note
that $y_2=H_+(f(y_{1}))\geq H_+(f(y_{0})) = y_{1}$, which follows from the
entry-wise monotonicity of $H_+$ and $y_{1}\geq y_{0}$. Then, by induction,
we obtain that $y_{k+1}=H_+(f(y_{k}))\geq y_{k}$ for $k\geq 0$. Now, notice
that $y_{k}\leq \vect{1}_n$ for $k\geq 0$, which let us conclude that
$(y_i(k))_k$ is a monotonically non-decreasing bounded sequence with upper
bound $1$. Then, $\lim_{k\to\infty}y_{k}=x^*$, with $x^*$ an equilibrium
point of the system in $Y$ and away from $\vect{1}_n$ due to
Lemma~\ref{lem:general}.
\end{proof}

\section{Analysis of higher-order models}

We extend the simplicial SIS model to the setting of multiple arbitrary
high-order interactions.

\begin{definition}[The general higher-order SIS model]
\label{def:simplicial_model_gen}
Assume $x\in[0,1]^n$, and let $\beta_1,\cdots,\beta_{n-1}>0$ and
$\gamma_i>0$, $i\in\until{n}$. Then, the \emph{general higher-order SIS
  model} is, for any $i\in\until{n}$,
\begin{equation*}
  \dot{x}_i=-\gamma_i x_i+\beta_1(1-x_i)\sum_{j=1}^n a_{ij}x_j+(1-x_i)\sum_{k=2}^{n-1}\beta_k \! \sum_{i_1,\dots,i_k=1}^n \!\! b_{ii_1\cdots i_k} x_{i_1}\cdots x_{i_k},
\end{equation*}
where $b_{ii_1\cdots i_k}\geq0$ for any $i\in\until{n}$ and
$k\in\{2,\cdots,n-1\}$, and $A=(a_{ij})$ is an arbitrary nonnegative matrix.
\end{definition}

We believe it is straightforward to extend the analysis of the simplicial
SIS model in Lemma~\ref{lem:general} to the general higher-order SIS model
in this definition. The reason is that the Lemma~\ref{lem:general}'s proof essentially depends 
on matrix $A$ and so is independent of 
any higher-order interaction; 
therefore,
we omit it here in the interest of brevity. Similarly,
under appropriate changes on the sufficient conditions that define each
behavioral domain, parallel results to Theorem~\ref{th:simplicial_main}
can 
be obtained. In the interest of brevity, we only focus 
on proving that a bistable domain also exists for arbitrary
higher-order interactions. 
For convenience, define the shorthand:
$$b_i^*:=\sum_{k=2}^{n-1}\beta_k\Big(\sum_{i_1,\dots,i_k=1}^nb_{ii_1\cdots i_k}\Big).$$

\begin{proposition}[The general higher-order SIS model and its different epidemiological domains]
\label{prop:bistable-g}
Consider the general higher-order SIS model (Definion~VI.1) with an irreducible $A\ge0$ and arbitrary $b_{ii_1\cdots i_k}\geq0$ for any $i\in\until{n}$ and $k\in\{2\cdots,n-1\}$. Define $\ybard\in\{0,1\}^n$ by $(\ybard)_i=1$ if
$b_i^*>0$ and $(\ybard)_i=0$ otherwise.
\begin{description}
\item[Disease-free domain:] If 
  \begin{equation*}
    \rho\Bigg(\beta_1\Gamma^{-1}A+\Gamma^{-1}\sum^{n-1}_{k=2}\beta_k\hat{\mathcal{B}}_k\Bigg)<1,    
%
  \end{equation*}
with $\hat{\mathcal{B}}_k\in\R^{n\times n}$ and its $ij$ entry being $\hat{\mathcal{B}}_{k,ij}=\sum^n_{\ell_2\dots\ell_k=1}b_{ij\ell_2\dots\ell_k}$. Then,  
  \begin{enumerate}[label=(\roman*)]
\item $\vect{0}_n$ is the unique equilibrium point in $[0,1]^n$,
\item $\vect{0}_n$ is globally exponentially stable in $[0,1]^n$ with
  Lyapunov function $V(x)=\norm{x}_{1,\diag(v)\Gamma^{-1}}=v^\top\Gamma^{-1}x$, where $v$ is the
  dominant left eigenvector of
  $\beta_1\Gamma^{-1}A+\Gamma^{-1}\sum^{n-1}_{k=2}\beta_k\hat{\mathcal{B}}_k$.
\setcounter{saveenum}{\value{enumi}}  
\end{enumerate}
\item[Bistable domain:] If $\beta_1\rho(\Gamma^{-1}A)<1$ and
\begin{equation*}
  \min_{i \text{ s.t. }b_i^*\neq 0} \Bigg(\frac{\beta_1}{\gamma_i}(A\ybard)_i 
  +  \sum_{k=2}^{n-1}\frac{\beta_k}{\gamma_i}\Big(\frac{n-2}{n-1}\Big)^{k-1}
  \sum_{i_1,\dots,i_k=1}^n \!\!  b_{ii_1\cdots i_k}\prod_{\ell=1}^k(\ybard)_{i_\ell} \Bigg) \geq n-1,
\end{equation*}
%
%
  then
\begin{enumerate}[label=(\roman*)]\setcounter{enumi}{\value{saveenum}}
\item\label{fact2-i} $\vect{0}_n$ is a locally exponentially stable equilibrium,
\item\label{fact2-ii} there exists an equilibrium point $x^*\gg\vect{0}_n$ such that
  $x^*_i\geq\frac{n-2}{n-1}$ for any $i$ such that $b_i^*\neq 0$,
  and
\item\label{fact2-iii} any such equilibrium point $x^*$ is locally exponentially stable.
    \setcounter{saveenum}{\value{enumi}}
\end{enumerate}
\item[Endemic domain:] If $\beta_1\rho(\Gamma^{-1}A)>1$, then
\begin{enumerate}[label=(\roman*)]\setcounter{enumi}{\value{saveenum}}
\item $\vect{0}_n$ is an unstable equilibrium,
\item there exists an equilibrium point $x^*\gg \vect{0}_n$ in $[0,1]^n$, and
\item if $\beta_2,\cdots,\beta_{n-1}$ is sufficiently small, then $x^*$ is unique in $(0,1]^n$
  and it is globally exponentially stable in
  $[0,1]^n\setminus\{\vect{0}_n\}$, with Lyapunov function
    $V(x)=\norm{x-x^*}_{\infty,\diag(x^*)^{-1}}$, $x\in\mathcal{X}$.
\end{enumerate}
\end{description}
Moreover, if $\beta_1\rho(\Gamma^{-1}A)<1$, then the system is either in the disease-free domain or in the bistable domain.
\end{proposition}
%
%
\begin{proof}
We only prove the results for the bistable domain. Consider the functions $H_+$ and $h_+$ introduced in the proof of
  Lemma~\ref{lem:general}. Let $\bar{A}:=\beta_1\Gamma^{-1}A$.  
  The proof for fact~\ref{fact2-i}   
  is 
  the same as
  in Theorem~\ref{th:simplicial_main}. 
  Now, we prove fact~\ref{fact2-ii}.     
  Define $Y=\setdef{y\in[0,1]^n}{\frac{n-2}{n-1}\ybard\leq y \leq
    \vect{1}_n}$.  Rewrite the second inequality assumption in the proposition 
  statement as $\theta\geq n-1$, where $\theta$ is a shorthand for the
  minimum term.  For a point $y\in Y$, we compute
  \begin{align}
    (H_+(y))_i \nonumber&=h_+ \Big( (\bar{A}y)_i+\sum_{k=2}^{n-1}\frac{\beta_k}{\gamma_i} \sum_{i_1,\dots,i_k=1}^n
    b_{ii_1\cdots i_k}\prod_{\ell=1}^ky_{i_\ell} \Big) \nonumber \\
    &\quad \geq h_+\Bigg(\frac{n-2}{n-1}((\bar{A}\ybar)_i +
    \sum_{k=2}^{n-1}\frac{\beta_k}{\gamma_i}
    \Big(\frac{n-2}{n-1}\Big)^{k-1} \sum_{i_1,\dots,i_k=1}^nb_{ii_1\cdots i_k}\prod_{\ell=1}^k(\ybard)_{i_\ell} )\Bigg),     \label{eq:aux_H_gen}
  \end{align}
  where the inequality follows from the monotonicity of the function $h_+$.
  Whenever $b_i^*\neq 0$, we can lower bound the expression
  in~\eqref{eq:aux_H_gen} by $h_+(\frac{n-2}{n-1}\theta)$; and whenever
  $b_i^*=0$, we can lower bound it by $h_+(0)=0$. Therefore, as in the
  proof of Theorem~\ref{th:simplicial_main}, we obtain
  $$H(y)\geq H_+ \Big(\frac{n-2}{n-1}\theta\ybard\Big)\geq
  \frac{n-2}{n-1}\ybard.$$ Then, following the proof for the bistable domain 
  of Theorem~\ref{th:simplicial_main}, we obtain that there exists an
  equilibrium point $y^*\in Y$ such that $y\gg\vect{0}_n$.

  Now we prove fact~\ref{fact2-iii}. Let $x^*\gg\vect{0}_n$ be an equilibrium satisfying $x^*\geq
  \frac{n-2}{n-1}\ybard$. Evaluating the Jacobian of the system at $x^*$,
  namely $Df(x^*)$, and after some algebraic work (similar to the one done
  in the proof of Theorem~\ref{th:simplicial_main}), we observe that
  $Df(x^*)$ is a Metzler matrix and, moreover, that
  \begin{equation*}
    \begin{aligned}
      &(Df(x^*)x^*)_i=-\beta_1(Ax^*)_ix^*_i+\sum_{k=2}^{n-1}\big((k-1)-kx^*_i\big)\beta_k\big(\sum_{i_1,\dots,i_k=1}^n
      b_{ii_1\cdots i_k}\prod_{\ell=1}^kx_{i_\ell}\big).
    \end{aligned}
  \end{equation*}
  First, if $b_i^*\neq0$, then $x^*_i\geq\frac{n-2}{n-1}$ and
  $(k-1-kx^*_i)\leq 0$, since $\frac{k-1}{k}\leq\frac{n-2}{n-1}\leq x^*$
  for $k\in\{2,\cdots,n-1\}$. In turn,
  \begin{equation*}
    \begin{aligned}
      (Df(x^*)x^*)_i
      \leq -\left(\beta_1\min_{s}
      \left(\sum_{j=1}^n a_{sj}x^*_j\right)\right)x_i^*.
    \end{aligned}
  \end{equation*}
  On the other hand, if $b_i^*=0$, then
  $(Df(x^*)x^*)_i=-\beta_1\left(\sum_{j=1}^n a_{ij}x^*_j\right)x^*_i$.
  Therefore, from these two cases and recalling that $A$ is irreducible, we
  have $Df(x^*)x^*\leq -dx^*$ for some $d>0$. Finally, since
  $x^*\gg\vect{0}_n$, \cite[Theorem 15.17]{FB:20} implies that $Df(x^*)$ is
  Hurwitz and, therefore, $x^*$ is locally exponentially stable. 
\end{proof}

\section{Numerical example}
\label{sec:numeric}
In Figure~\ref{f:sim1}, we present two numerical examples of the behavior
of the simplicial SIS model. First, we verify the existence of a parameter
region under which the sufficient conditions of
Theorem~\ref{th:simplicial_main} cannot be applied. We can readily observe
the transition from the disease-free domain to the bistable domain as we
increase $\beta_2$ for a fixed $\beta_1$, as mentioned in
Remark~\ref{rem-main-th}. Also, notice that the sufficient condition for
determining the endemic domain in Theorem~\ref{th:simplicial_main} is
tight. We also remark that the sufficient condition for determining the
bistable region captures most of the true parameter region in these
simulations.

From our numerical simulations we propose the following conjectures, which
are consistent with the behavior observed in the scalar model.
\begin{conjectures}[Behaviors in the bistable and endemic domains]
  For the simplicial SIS model,
  \begin{enumerate}
  \item in the bistable domain, at fixed $\beta_2$, the domain of
    attraction of the disease-free equilibrium $x^*=\vect{0}_n$ decreases
    as $\beta_1$ increases. Once $\beta_1=\frac{1}{\rho(\Gamma^{-1}A)}$, a
    bifurcation occurs and the origin becomes an unstable equilibrium point
    in the endemic domain;
    
  \item in the endemic domain, the endemic equilibrium is unique and
    globally stable for any value of $\beta_2$.
\end{enumerate}
\end{conjectures}

%

\begin{figure}[ht]
  \centering
  \includegraphics[width=0.6\linewidth]{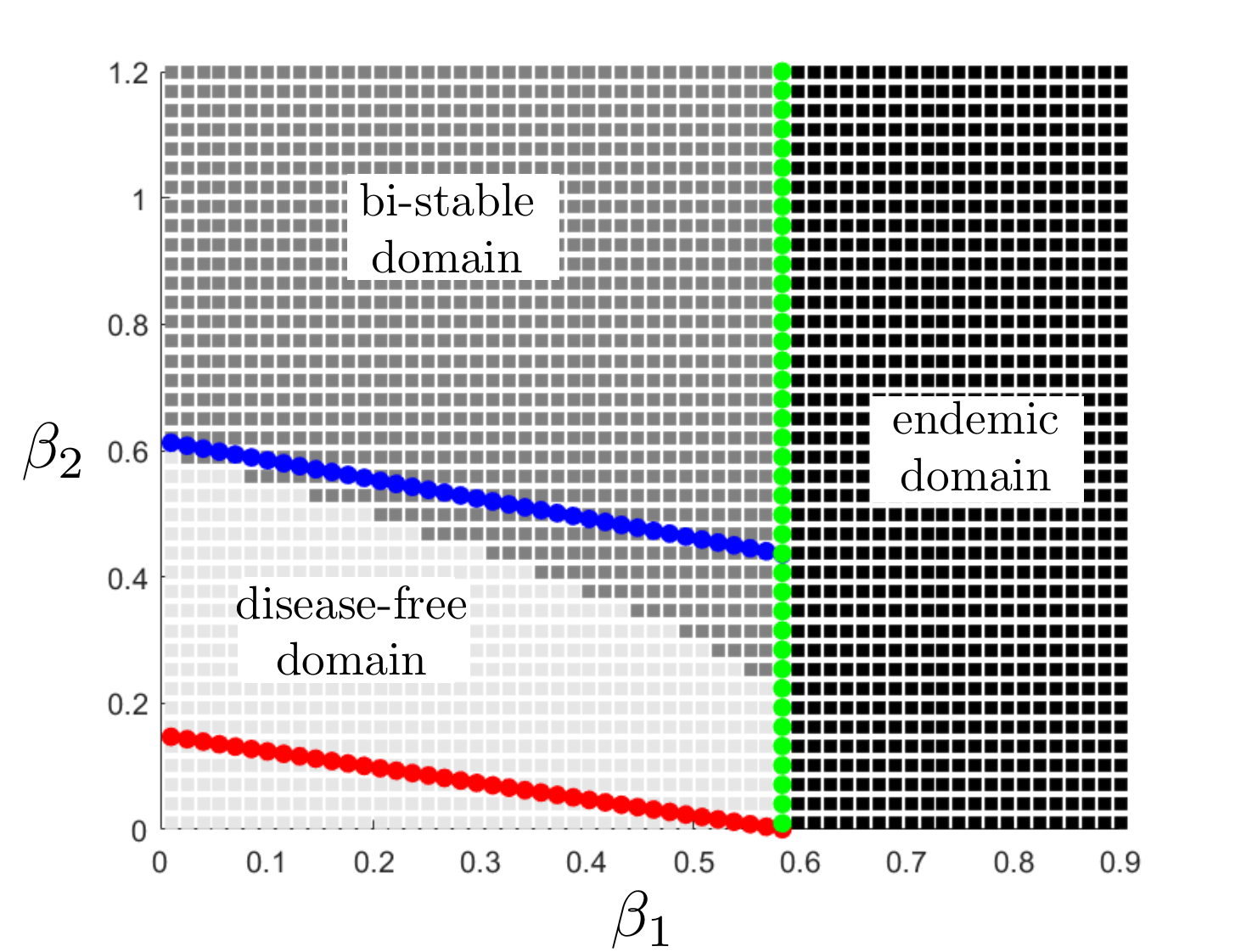}
  \includegraphics[width=0.6\linewidth]{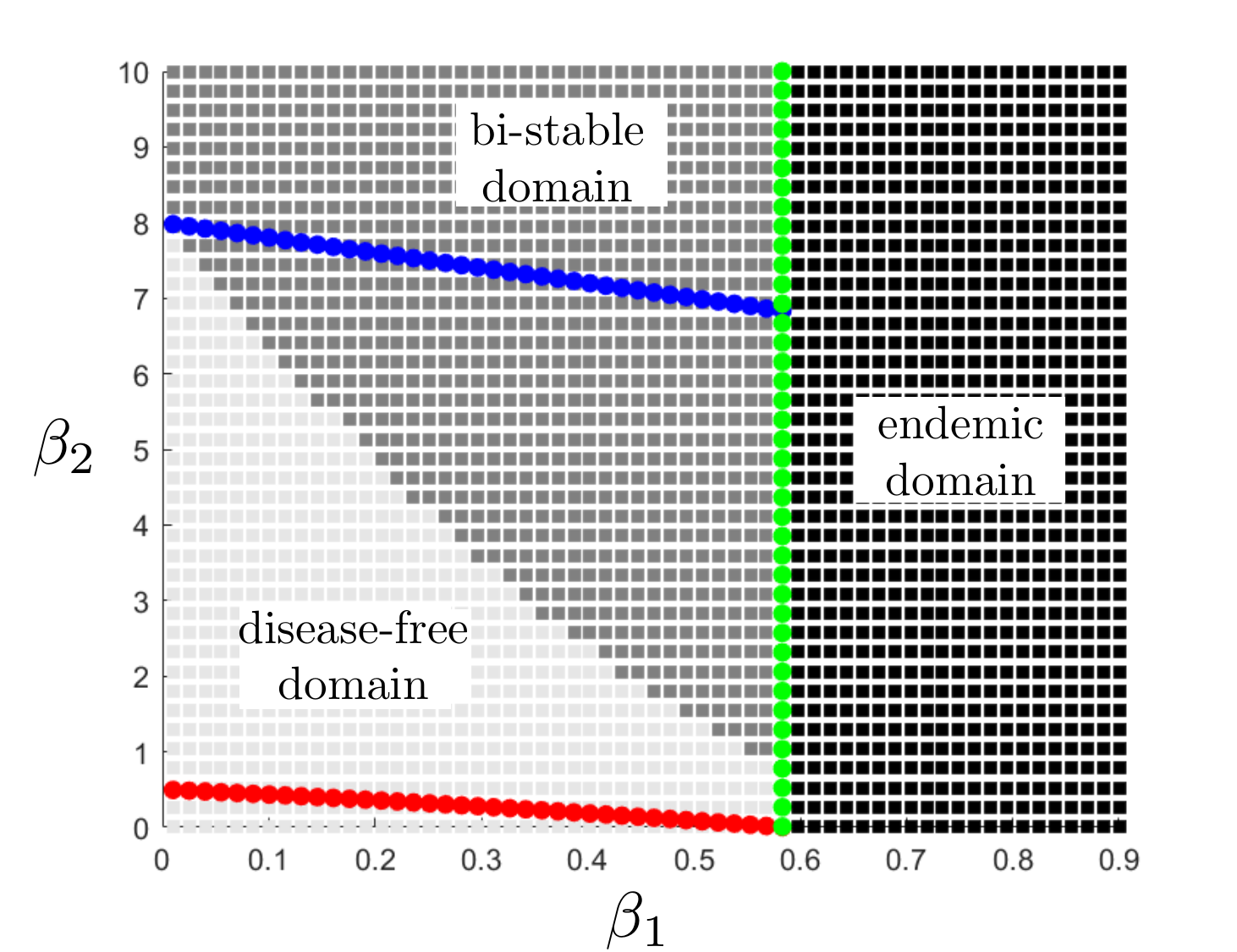}
  \caption{Consider the simplicial SIS model with its parameters
    (Definition~\ref{def:simplicial_model}). For the upper figure, we
    randomly generated and fixed six irreducible matrices
    $A\in\{0,1\}^{5\times{5}}$ and $B_i\in\{0,1\}^{5\times{5}}$,
    $B_i\neq\vect{0}_n$, $i\in\until{5}$; and fixed $\Gamma=2I_5$.  The
    light-gray/gray/black region corresponds to the
    disease-free/bistable/endemic domain from the simulation.  Regarding the
    sufficient conditions established by Theorem~\ref{th:simplicial_main},
    all the region to the right of the green line correspond to the endemic
    domain, all the region above the blue line and left to the green line
    to the bistable domain, and all the region below the red curve to the
    disease-free domain.  For the lower figure, we considered the same
    settings as in the upper figure, but with the difference that this time
    we set $B_i=\vect{0}_{n\times{n}}$ for $i\in\{2,\cdots,5\}$.  }
  \label{f:sim1}
\end{figure}

\section{Conclusion}
\label{sec:concl} 
In this paper, we formally analyze the simplicial SIS model and establish
its different behavioral domains. As seen in a previous 
scalar model, we show the existence of the bistable domain and its
possible transition from the disease-free domain by changing the model
parameters. This feature makes our model qualitatively different from the
classical multi-group SIS model. We also show that the bistable domain
exists for any multi-group SIS model with higher-order interactions.

As future work, we plan to study control strategies for the mitigation of
the epidemic 
in the simplicial SIS model; 
e.g.,   
how to
drive the system to the origin whenever it is in the bistable domain. 
More
generally, we also plan to study the 
aggregation of 
higher order
interaction terms in other epidemiological models, 
where 
we believe our approach based
on Coppel's inequalities can also be useful. 
Finally, 
it is relevant to provide a more comprehensive characterization of the model parameters $\beta_1$ and $\beta_2$, and thus prove the tight transition between the disease-free and the bistable domains illustrated by our simulations.

\section*{Acknowledgment}
We thank Dr.\ Saber Jafarpour for insightful discussions about matrix
measures and Coppel's inequalities, including proving
Theorem~\ref{theorem:append} and relate it to~\cite{MV:02}. 

\bibliographystyle{plainurl+isbn}
\bibliography{alias,Main,FB}

\end{document}